\documentclass{amsart}
\usepackage{graphicx}
\usepackage{verbatim}
\usepackage{textcomp}
\newtheorem{thm}{Theorem}[section]

\newtheorem{lem}[thm]{Lemma}

\theoremstyle{definition}

\theoremstyle{remark}
\newtheorem{rem}[thm]{Remark}
\numberwithin{equation}{section}

\newcommand{\Real}{\mathbb R}

\newcommand{\To}{\longrightarrow}

\title{A Bernstein Type Theorem For Self-similar Shrinkers}
\author{Lu Wang}%
\address{Mathematics Department, 
Massachusetts Institute of Technology, 
77 Massachusetts Avenue, Cambridge, MA 02139.}
\email{luwang@math.mit.edu}

\begin{document}
\begin{abstract}
In this note, we prove that smooth self-shrinkers in $\Real^{n+1}$, that are entire graphs, are hyperplanes. Previously Ecker and Huisken showed that smooth self-shrinkers, that are entire graphs and have at most polynomial growth,  are hyperplanes. The point of this note is that no growth assumption at infinity is needed.
\end{abstract}
\maketitle

\section{Introduction}
In this note, we show that smooth self-shrinkers in $\Real^{n+1}$, that are entire graphs, are hyperplanes. A smooth one-parameter family of hypersurfaces, $F:(0,T)\times M^n\To\Real^{n+1}$, is moving by mean curvature, if 
\begin{equation}
\label{Eqn7}
\frac{dF}{dt}=-H\textbf{n},
\end{equation}
where $\textbf{n}$ is the unit normal of $M_t=F(t,M)$ and $H=\text{div}(\textbf{n})$.
Self-shrinkers represent a special class of solutions of equation $(\ref{Eqn7})$, which do not change shape under mean curvature flow. Namely, for self-shrinkers, a later time slice is a scaled down copy of an earlier slice.
Self-shrinkers play a crucial role in studying the singularities of mean curvature flow. For instance, in \cite{H1} and \cite{H2}, Huisken discovered a monotonicity formula for mean curvature flow and proved that, around the singular point $(x_0,t_0)$, if the blow-up rate of curvature is bounded above by a multiple of $1/\sqrt{t_0-t}$, then the rescaled hypersurfaces moving by mean curvature converge to a self-shrinker smoothly. In the recent work of Colding and Minicozzi, they showed that the only smooth embedded entropy stable self-shrinkers with polynomial volume growth in $\Real^{n+1}$ are hyperplanes, $n$-spheres and cylinders $\Real^k\times S^{n-k}$, $0<k<n$, and classified the generic singularities of mean curvature flow; see \cite{CM1}. On the other hand, self-shrinkers are minimal hypersurfaces in $\Real^{n+1}$ under the conformal metric $g_{ij}=\exp(-|\overrightarrow{x}|^2/2n)
\delta_{ij}$; see \cite{An}, \cite{CM1} and \cite{CM2}. In minimal hypersurface theory, the Bernstein Theorem is one of the most fundamental theorems, and has many important applications, such as uniqueness and regularity of minimal hypersurfaces. Thus it is natural to ask whether there is a Bernstein type theorem for self-shrinkers. Ecker and Huisken studied mean evolution of entire graphs in a series of papers beginning with \cite{EH1} in 1989. In particular, they proved in the appendix of \cite{EH1} that smooth self-shrinkers in $\Real^{n+1}$, that are entire graphs and have at most polynomial growth, are hyperplanes. Later, in \cite{EH2}, they derived various interior estimates for mean curvature flow and proved the existence of smooth mean envolution of graphs with only locally Lipschitz initial data. Also, in \cite{CM3}, Colding and Minicozzi proved sharp gradient and area estimates for graphs moving by mean curvature. In this note, we study self-shrinkers using the $L$-stability operator, which was introduced by Colding and Minicozzi in \cite{CM1} and \cite{CM2}. We give an elementary proof of a Berntein type theorem for self-shrinkers in $\Real^{n+1}$ without assumption of the growth at infinity and without using gradient or curvature estimates for mean curvature flow.\footnote[1]{It follows from Lemma \ref{Lem3} in the next section and the interior gradient estimate in \cite{EH2} and \cite{CM3} that the gradient of entire graphical self-shrinkers is bounded. Thus the proposition in the appendix of \cite{EH1} applies. However, the point of this note is to give an elementary proof of a Bernstein type theorem for self-shrinkers, which parallels the Bernstein theorem for minimal hypersurfaces.} In contrast to the Bernstein theorem for minimal hypersurfaces, which is true only for $n\leq 6$ (see \cite{S}, \cite{BDG} and \cite{BDM}), our Bernstein type theorem for self-shrinkers holds for any $n$. The proof in section $2$ is very simple in any dimension, while the proof of the Bernstein theorem for minimal hypersurfaces is complicated even in relative low dimension; see \cite{SSY}. The reason behind this is that smooth self-shrinkers in $\Real^{n+1}$, that are entire graphs, have polynomial volume growth as minimal hypersurfaces (although the orders of volume growth are different), and a weighted stability inequality $(\ref{Eqn8})$ with weight $\exp(-|\overrightarrow{x}|^2/4)$, which makes the right hand side of $(\ref{Eqn8})$ tending to zero in any dimension by choosing appropriate cut-off functions.

Throughout we use the subscripts $x_1,\dots,x_n$ and $t$ to denote derivatives of functions with respect to $x_1,\dots,x_n$ and $t$; $\overrightarrow{x}$ is the position vector in $\Real^{n+1}$; $\omega_n$ is half of the volume of the unit $n$-sphere in $\Real^{n+1}$; we define
\begin{eqnarray*}
v_i&=&(0,\dots,0,1,0,\dots,0),\\
&&\quad\quad\quad\quad ith
\end{eqnarray*}
where $i$ is an integer and $i\in\{1,\dots,n+1\}$.
The main result of this note is that:
\begin{thm}
\label{Thm1}
Suppose that the smooth function $u(x_1,\dots,x_n):\Real^n\To\Real$ satisfies the self-shrinker equation:
\begin{equation}
\label{Eqn1}
\emph{div}\left(\frac{Du}{\sqrt{1+|Du|^2}}
\right)=\frac{x_1u_{x_1}+\dots+x_nu_{x_n}-u}{2\sqrt{1+|Du|^2}},
\end{equation}
where $Du=(u_{x_1},\dots,u_{x_n})$ and $|Du|^2=u_{x_1}^2+\dots+u_{x_n}^2$. Then $u=a_1x_1+\dots+a_nx_n$ for some constants $a_1,\dots,a_n\in\Real$. 
\end{thm}
\begin{rem}
\label{Rem1}
We say that the graph of $u$ in Theorem \ref{Thm1} is an entire graphical self-shrinker in $\Real^{n+1}$ and denote $\text{Graph}_u$ by  $\Sigma$. Note that the left hand side of equation (\ref{Eqn1}) is minus the mean curvature $H$ of $\Sigma$ and the right hand side of equation (\ref{Eqn1}) is $-\langle\overrightarrow{x},\textbf{n}\rangle/2$, where $\textbf{n}$ is the upward unit normal of $\Sigma$. Therefore equation (\ref{Eqn1}) is equivalent to 
\begin{equation}
\label{Eqn3} H=\frac{1}{2}\langle\overrightarrow{x},\textbf{n}\rangle.
\end{equation}
\end{rem}
\section{Proof of Theorem \ref{Thm1}}
First, we prove a weighted stability inequality for smooth self-shrinkers in $\Real^{n+1}$ that are entire graphs. In \cite{CM1} and \cite{CM2}, Colding and Minicozzi introduced the operator
\begin{equation} 
L=\Delta_\Sigma
-\frac{1}{2}\langle\overrightarrow{x},\nabla_\Sigma\rangle
+|A|^2+\frac{1}{2},
\end{equation}
which is given by the linearization of the self-shrinker equation. The weighted stability inequality in Lemma \ref{Lem1} below is equivalent to that $-(L-\frac{1}{2})\geq 0$. On the other hand, they showed that $LH=H$ in Lemma $5.5$ of \cite{CM1}. Let $\eta$ be a smooth and compactly supported function on $\Real^{n+1}$. And suppose that $\eta$ is identically one on $B_R$ and cuts off linearly to zero on $B_{R+1}\setminus B_R$, where $B_R$ is the open ball in $\Real^{n+1}$ centered at the origin with radius $R>0$. Thus, similar to the computation in the appendix of \cite{CM2}, we get
\begin{equation*}
-\frac{1}{2}\int_{\Sigma}\eta^2H^2\text{e}^{-\frac{|\overrightarrow{x}|^2}{4}}
\leq-\int_{\Sigma}\eta HL(\eta H)\text{e}^{-\frac{|\overrightarrow{x}|^2}{4}}\leq -\int_{B_R\cap\Sigma}H^2\text{e}^{-\frac{|\overrightarrow{x}|^2}{4}}+\int_{\Sigma\setminus B_R}H^2\text{e}^{-\frac{|\overrightarrow{x}|^2}{4}}.
\end{equation*} 
Hence, if self-shrinkers have polynomial volume growth and let $R\To\infty$, then it follows from the monotone convergence theorem that $H=0$. And the only smooth embedded minimal cones in $\Real^{n+1}$ are hyperplanes.
\begin{lem}
\label{Lem1}
Let $\eta$ be a smooth compactly supported function on $\Real^{n+1}$. Then
\begin{equation}
\label{Eqn8}
\int_\Sigma\eta^2|A|^2\emph{e}^{-\frac{|\overrightarrow{x}|^2}{4}}\leq \int_\Sigma|\nabla_\Sigma\eta|^2
\emph{e}^{-\frac{|\overrightarrow{x}|^2}{4}},
\end{equation}
where $A=(a_{ij})$ is the second fundamental form of $\Sigma$ in $\Real^{n+1}$ and $\nabla_\Sigma$ is the gradient of a function with respect to $\Sigma$.
\end{lem}
\begin{proof}
Let $f=\langle\textbf{n},v_{n+1}\rangle$. Colding and Minicozzi showed that $Lf=\frac{1}{2}f$ in Lemma $5.5$ of \cite{CM1}. For self-containedness, we include the proof here.  Indeed, at any point $\overrightarrow{x}\in\Sigma$, choose a local geodesic frame $e_1,\dots,e_n$, that is, $\langle e_i,e_j\rangle=\delta_{ij}$ and $\nabla^\Sigma_{e_i}e_j(\overrightarrow{x})=0$. Thus, 
\begin{eqnarray*}
\nabla_\Sigma f&=&\sum_{i=1}^n\langle\nabla_{e_i}\textbf{n},v_{n+1}\rangle e_i
=\sum_{i,j=1}^n-a_{ij}\langle e_j,v_{n+1}\rangle e_i,\\
\Delta_\Sigma f&=&\sum_{i=1}^n
\langle\nabla_{e_i}\nabla_{e_i}\textbf{n},v_{n+1}\rangle
=\sum_{i,j=1}^n
-a_{ij;i}\langle e_j,v_{n+1}\rangle-a_{ij}\langle\nabla_{e_i}e_j,v_{n+1}\rangle\\
&=&\langle\nabla_\Sigma H,v_{n+1}\rangle-|A|^2\langle\textbf{n},v_{n+1}\rangle,
\end{eqnarray*}
where $H=-\sum_{i=1}^na_{ii}$ is the mean curvature of $\Sigma$. Since $\Sigma$ is self-shrinker, i.e. $H=\frac{1}{2}\langle\overrightarrow{x},\textbf{n}\rangle$, thus,
\begin{equation}
\nabla_\Sigma H=\frac{1}{2}\sum_{i=1}^n
\langle\overrightarrow{x},\nabla_{e_i}\textbf{n}\rangle e_i
=-\frac{1}{2}\sum_{i,j=1}^na_{ij}
\langle\overrightarrow{x},e_j\rangle e_i.
\end{equation}
Hence,
\begin{equation}
\Delta_\Sigma f=\frac{1}{2}\langle\overrightarrow{x},\nabla_\Sigma f\rangle
-|A|^2 f.
\end{equation} 
Note that the upward unit normal of $\Sigma$ is 
\begin{equation}
\textbf{n}=\frac{(-Du,1)}
{\sqrt{1+|Du|^2}},
\end{equation}
and thus $f=\langle\textbf{n},v_{n+1}\rangle
=1/\sqrt{1+|Du|^2}>0$. 
Hence, the function $g=\log f$ is well defined and $g$ satisfies the differential equation
\begin{equation}
\label{Eqn2}
\Delta_\Sigma g-\frac{1}{2}\langle\overrightarrow{x},\nabla_\Sigma g\rangle+|\nabla_\Sigma g|^2+|A|^2=0.
\end{equation}
Multiplying by $\eta^2\text{e}^{-\frac{|\overrightarrow{x}|^2}{4}}$ on both sides of equation (\ref{Eqn2}) and integrating over $\Sigma$, gives
\begin{eqnarray*}
0&=&\int_\Sigma\eta^2\text{div}_\Sigma
\left(\text{e}^{-\frac{|\overrightarrow{x}|^2}{4}}
\nabla_\Sigma g\right)+\int_\Sigma\eta^2\left(|\nabla_\Sigma g|^2+|A|^2\right)\text{e}^{-\frac{|\overrightarrow{x}|^2}{4}}\\
&=&-\int_\Sigma 2\eta\langle\nabla_\Sigma\eta,\nabla_\Sigma g\rangle \text{e}^{-\frac{|\overrightarrow{x}|^2}{4}}
+\int_\Sigma\eta^2\left(|\nabla_\Sigma g|^2+|A|^2\right)
\text{e}^{-\frac{|\overrightarrow{x}|^2}{4}}\\
&\geq&-\int_\Sigma\left(\eta^2|\nabla_\Sigma g|^2+|\nabla_\Sigma\eta|^2\right)
\text{e}^{-\frac{|\overrightarrow{x}|^2}{4}}
+\int_\Sigma\eta^2\left(|\nabla_\Sigma g|^2+|A|^2\right)
\text{e}^{-\frac{|\overrightarrow{x}|^2}{4}}\\
&\geq&\int_\Sigma\left(-|\nabla_\Sigma\eta|^2+\eta^2|A|^2\right)
\text{e}^{-\frac{|\overrightarrow{x}|^2}{4}}.
\end{eqnarray*}
\end{proof}
Second, we study the volume growth of entire graphical self-shrinkers. Let $B_R$ be the open ball in $\Real^{n+1}$ centered at the origin with radius $R$, and $\mathcal{B}_R$ be the open ball in $\Real^n$ centered at the origin with radius $R$. Also, we define $M_R=\sup_{\mathcal{B}_R}|u|$. We show that
\begin{lem}
\label{Lem2}
There exists a constant $C_0>0$, depending only on $n$, such that 
\begin{equation}
\emph{Vol}(\Sigma\cap B_R)\leq C_0R^n(1+R^n+M_R^2),
\end{equation} 
where \emph{Vol} stands for volume.
\end{lem} 
\begin{proof}
First, using the pull back of \textbf{n} induced by the projection $\pi:\mathcal{B}_R\times\Real\To\mathcal{B}_R$, we extend the vector field \textbf{n} on the cylinder $\mathcal{B}_R\times\Real$. Let $\omega$ be the $n$-form on the cylinder $\mathcal{B}_R\times\Real$ given by that for any $X_1\dots,X_n\in\Real^{n+1}$,
\begin{equation}
\label{Eqn9}
\omega(X_1,\dots,X_n)=\det(X_1,\dots,X_n,\textbf{n}).
\end{equation}
Then, in coordinates $(x_1,\dots,x_{n+1})$, we have
\begin{equation}
\omega=\frac{dx_1\wedge\cdots
\wedge dx_n
-\sum_{i=1}^n
(-1)^{in}u_{x_i}dx_{i+1}\wedge\cdots\wedge dx_{n+1}\wedge dx_1\wedge\cdots\wedge dx_{i-1}}
{\sqrt{1+|Du|^2}},
\end{equation}
and 
\begin{equation}
d\omega=(-1)^{n+1}\text{div}
\left(\frac{Du}{\sqrt{1+|Du|^2}}\right)
dx_1\wedge\cdots\wedge dx_{n+1}
\end{equation}
\begin{equation*}
\quad\quad\quad\quad=(-1)^{n+1}\frac{x_1u_{x_1}+\cdots+x_nu_{x_n}-u}
{2\sqrt{1+|Du|^2}}
dx_1\wedge\cdots\wedge dx_{n+1}.
\end{equation*}
By the H\"{o}lder inequality,
\begin{equation}
 |d\omega|\leq\frac{1}{2}\sqrt{x_1^2+\cdots+x_n^2+u^2},
\end{equation}
and by (\ref{Eqn9}), given any orthogonal unit vectors $X_1,\dots,X_n$ at a point $(x_1,\dots,x_{n+1})$,
\begin{equation}
|\omega(X_1,\dots X_n)|\leq 1,
\end{equation} 
where the equality holds if and only if 
\begin{equation}
X_1,\dots,X_n\in T_{(x_1,\dots,x_n,u(x_1,\dots,x_n))}\Sigma.
\end{equation}
For minimal graphs, where $d\omega=0$, such an $\omega$ is called a calibration; see page $3$ in \cite{CM4}. 
Note that $\partial B_R\cap\Sigma$ divides $\partial B_R$ into two components and let $\partial B_R^1$ be the component which has volume at most equal to $\omega_nR^n$. Let $\Omega$ be the region enclosed by $\partial B_R^1\cup\Sigma$. Hence, by Stokes' Theorem,\footnote[1]{we choose the orientation of $\Sigma$ to be compatible with the upward unit normal and the orientation of $\partial B_R$ to be compatible with the inward unit normal. Thus the orientation of $\Omega$ is chosen such that the orientation of $\partial\Omega$ induced from $\Omega$ coincides with that we just defined above.} 
\begin{eqnarray*}
\text{Vol}(\Sigma\cap B_R)&\leq&\int_{\Sigma\cap\partial\Omega}\omega
=\int_{\partial B_R^1}\omega+\int_\Omega d\omega\\
&\leq& \text{Vol}(\partial B_R^1)
+\frac{1}{2}\int_\Omega\sqrt{x_1^2+\cdots+x_n^2+u^2}.
\end{eqnarray*}
Since $\Omega$ is contained in the cylinder $\tilde{\Omega}=\mathcal{B}_R\times [-\max\{R,M_R\},\max\{R,M_R\}]$, we conclude that 
\begin{equation}
\text{Vol}(\Sigma\cap B_R)\leq \omega_n R^n+\frac{1}{2}\int_{\tilde{\Omega}}(R+M_R)\leq\omega_n R^n
+\omega_n R^n(R+M_R)^2.
\end{equation}
Therefore, 
\begin{equation}
\text{Vol}(\Sigma\cap B_R)\leq 2\omega_n R^n(1+R^2+M_R^2),
\end{equation}
and $C_0=2\omega_n$.
\end{proof}
Third, we use the maximum principle for the mean curvature flow to bound the $L^\infty$ norm of $u$ on $\mathcal{B}_R$.
\begin{lem}
\label{Lem3}
Suppose that $R>1$. Then there exists a constant $C_1>0$, depending on $n$ and $M_{2\sqrt{n}}$, such that $M_R\leq C_1R$. In particular, entire graphical self-shrinkers have polynomial volume growth.
\end{lem}
\begin{proof}
Define 
\begin{equation}
w(x_1,\dots,x_n,t)=\sqrt{R^2+1-t}\cdot
u\left(x_1/\sqrt{R^2+1-t},\dots, x_n/\sqrt{R^2+1-t}\right), 
\end{equation}
where $t\in[0,R^2]$. We derive the evolution equation of $w$:
\begin{equation}
\label{Eqn4}
\frac{dw}{dt}=\sqrt{1+|Dw|^2}\cdot
\text{div}\left(\frac{Dw}{\sqrt{1+|Dw|^2}}\right).
\end{equation}
Thus $\{\Sigma_t=\text{Graph}_{w(\cdot,t)}\}_{t>0}$ is a smooth family of hypersurfaces in $\Real^{n+1}$ moving by mean curvature (after composing with appropriate tangential diffeomorphisms). Similar to the arguments of Lemma $3$ in \cite{CM3} and \cite{E}, we construct suitable open balls as barriers. Let $\rho>0$ be some constant to be chosen later and $a^+=\sup_{\mathcal{B}_{\rho R}}w(x_1,\dots,x_n,0)+\rho R+1$. Consider the open ball $B^+_0$ centered at 
$(0,\dots,0,a^+)$ with radius $\rho R$. It is easy to check that $x_1^2+\cdots+x_n^2+(w(x_1,\dots,x_n,0)-a^+)^2>\rho R$, for any $(x_1,\dots,x_n)\in\Real^n$. Thus $B^+_0$ and $\Sigma_0$ are disjoint. Let $B^+_t$ be the open ball in $\Real^{n+1}$ centered at $(0,\dots,0,a^+)$ with radius $R_1^+=\sqrt{\rho^2R^2-2nt}$. Thus $\{\partial B^+_t\}_{t>0}$ is a smooth family of hypersurfaces moving by mean curvature and it shrinks to its center at $T=\rho^2R^2/2n$. We choose $\rho^2>2n+1$ to guarantee that $\partial B^+_t$ is not contained in the cylinder $\mathcal{B}_R\times\Real$, for $t\in[0,R^2]$. By the maximum principle for the mean curvature flow, $\Sigma_t$ and $\partial B^+_t$ are always disjoint for $t\in[0,R^2]$. Indeed, assume that 
$T_0$ is the first time that $\Sigma_{T_0}$ and $\partial B_{T_0}^+$ are not disjoint. Note that at every time $t$, the distance between $\Sigma_t$ and $\partial B_t^+$ can be achieved by a straight line segment perpendicular to both $\Sigma_t$ and $\partial B^+_t$. And outside $\mathcal{B}_{2\rho R}\times\Real$, the distance between $\Sigma_t$ and $\partial B_t^+$ is larger than $\rho R$. Thus at $T_0$, $\Sigma_{T_0}$ touches $\partial B_{T_0}^+$ at some point $\overrightarrow{x_0}\in\Real^{n+1}$, and for $t$ close to $T_0$, $\Sigma_t$ and $\partial B_t^+$ can be written as graphs over the tangent hyperplane of $\Sigma_{T_0}$ at $\overrightarrow{x_0}$ in a small neighborhood of $\overrightarrow{x_0}$. Thus for $t$ close to $T_0$, the evolution equations of the corresponding graphs are locally uniformly parabolic. Hence, the assumption violates the maximum principle for uniformly parabolic partial differential equations. This is a contradiction.
Therefore, we get the upper bound for $w$ at time $t=R^2$:
\begin{eqnarray*}
\sup_{\mathcal{B}_R}w(x_1,\dots,x_n,R^2)&\leq& a^+-\sqrt{\rho^2-2n-1}\cdot R\\
&\leq&\sup_{\mathcal{B}_{\rho R}}w(x_1,\dots,x_n,0)+\rho R+1-\sqrt{\rho^2-2n-1}\cdot R\\
&\leq&\sup_{\mathcal{B}_{\rho R}}w(x_1,\dots,x_n,0)+\sqrt{2n+1}R+1.
\end{eqnarray*}
Similarly, define $a^-=\inf_{\mathcal{B}_{\rho R}}w(x_1,\dots,x_n,0)-\rho R-1$ and 
compare $\Sigma_t$ with $\partial B_t^-$, which is centered at $(0,\dots,0,a^-)$ with radius $R_1^-=\sqrt{\rho^2 R^2-2nt}$. Therefore,
\begin{equation}
\inf_{\mathcal{B}_R}w(x_1,\dots,x_n,R^2)\geq
\inf_{\mathcal{B}_{\rho R}}w(x_1,\dots,x_n,0)-\sqrt{2n+1}R-1.
\end{equation}
In sum,
\begin{equation}
\label{Eqn5}
\sup_{\mathcal{B}_R}|w|(x_1,\dots,x_n,R^2)\leq \sup_{\mathcal{B}_{\rho R}}|w|(x_1,\dots,x_n,0)+\sqrt{2n+1}R+1.
\end{equation}
Note that $w(x_1,\dots,x_n,R^2)=u(x_1,\dots,x_n)$ and 
\begin{equation}
w(x_1,\dots,x_n,0)=\sqrt{R^2+1}\cdot u(x_1/\sqrt{R^2+1},\dots,x_n/\sqrt{R^2+1}).
\end{equation}
Thus, by inequality (\ref{Eqn5}) and the assumption that $R>1$, we conclude that  
\begin{equation}
\label{Eqn6}
\sup_{\mathcal{B}_R}|u|\leq 2(\sup_{\mathcal{B}_\rho}|u|+\sqrt{2n+1})R.
\end{equation}
Hence, choosing $\rho=2\sqrt{n}$ and $C_1=2(\sup_{\mathcal{B}_\rho}|u|+\sqrt{2n+1})$, gives $M_R\leq C_1R$ .
\end{proof}
Finally, we choose a sequence of $R_j\to \infty$ and a sequence of smooth cut-off functions $\eta_j$ , which satisfies that $0\leq \eta_j\leq 1$, $\eta_j$ is $1$ inside $B_{R_j}$ and vanishes outside $B_{R_j+1}$, and $|\nabla_\Sigma\eta_j|\leq |D\eta_j|\leq 2$. By the polynomial volume growth of $\Sigma$, we get that as $j\To\infty$,
\begin{equation}
\int_\Sigma|\nabla_\Sigma \eta_j|^2\text{e}^{-\frac{|\overrightarrow{x}|^2}{4}}\leq \int_{\Sigma\cap (B_{R_j+1}\setminus B_{R_j})}2\text{e}^{-\frac{|\overrightarrow{x}|^2}{4}}
\To 0.
\end{equation}
Hence, by the weighted stability inequality for $\Sigma$ and monotone convergence theorem, we conclude that 
\begin{equation}
\int_\Sigma|A|^2\text{e}^{-\frac{|\overrightarrow{x}|^2}{4}}=0.
\end{equation}
Therefore, $|A|=0$ and $u=a_1x_1+\cdots+a_nx_n$ for some constants $a_1,\dots,a_n\in\Real$.\\

\textbf{Acknowledgement} The author would like to thank her advisor Prof. Tobias Colding for continuous encouragement and many useful suggestions for the revision of the manuscript of this note.
\bibliographystyle{plain}
\bibliography{BSelfShrinker}

\end{document}